\documentclass[12pt]{amsart}
\usepackage{amsfonts, amsmath, amscd}
\usepackage[psamsfonts]{amssymb}
\usepackage{amssymb}

\usepackage[colorlinks]{hyperref}

\newtheorem{theorem}{Theorem}[section]
\newtheorem{lemma}[theorem]{Lemma}
\newtheorem{proposition}[theorem]{Proposition}

\newtheorem{corollary}[theorem]{Corollary}
\theoremstyle{definition}

\newtheorem{remark}[theorem]{Remark}

\numberwithin{equation}{section}

\newcommand{\card}{\operatorname{card}}

\begin{document}
\title{ On non-archimedean Gurari\v{i} spaces }
\author{J. K\c akol}
\address{Adam Mickiewicz University, Pozna\'n, Poland and Institute of
Mathematics, Czech Academy of Sciences, Czech Republic}
\thanks{The first named author was supported by Generalitat Valenciana,
Conselleria d'Educaci\'{o}, Cultura i Esport, Spain, Grant PROMETEO/2015/058
and by the GA\v{C}R project 16-34860L and RVO: 67985840.}
\author{W.~Kubi\'{s}}
\address{Institute of Mathematics, Czech Academy of Sciences, Czech Republic
and Cardinal Stefan Wyszy\'nski University, Warsaw, Poland}
\thanks{The second author was supported by the GA\v{C}R project 16-34860L
and RVO: 67985840.}
\author{A. Kubzdela}
\address{Institute of Civil Engineering, University of Technology, Pozna\'n,
Poland}

\keywords{Non-archimedean Banach spaces, Isometric embedding, Universal disposition}
\subjclass[2010]{46S10}
\maketitle

\begin{abstract}
Let $\mathcal{U}_{FNA}$ be the class of all non-archimedean
finite-dimensional Banach spaces. A non-archimedean Gurari\v{i} Banach space 
$\mathbb{G}$ over a non-archimedean valued field $\mathbb{K}$ is
constructed, i.e. a non-archimedean Banach space $\mathbb{G}$ of countable
type which is of \emph{almost universal disposition} for the class $\mathcal{U}
_{FNA}$. This means: for every isometry $g:X\rightarrow Y$, where $Y\in 
\mathcal{U}_{FNA}$ and $X$ is a subspace of $\mathbb{G}$, and every $%
\varepsilon \in (0,1)$ there exists an $\varepsilon $-isometry $f:Y\rightarrow 
\mathbb{G}$ such that $f(g(x))=x$ for all $x\in X$. We show that all
non-archimedean Banach spaces of almost universal disposition for the class 
$\mathcal{U}_{FNA}$ are $\varepsilon$-isometric. Furthermore, all non-archimedean
Banach spaces of almost universal disposition for the class $\mathcal{U}%
_{FNA}$ are isometrically isomorphic if and only if $\mathbb{K}$ is
spherically complete and $\left\{ \left\vert \lambda \right\vert :\lambda
\in \mathbb{K}\backslash \left\{ 0\right\} \right\} =\left( 0,\infty \right)$.
\end{abstract}

\section{Introduction}

In 1966 Gurari\v{i} constructed a separable (real) Banach space $\mathbb{G}$
of \emph{almost universal disposition} for finite-dimensional spaces (called
later the \emph{Gurari\v{i} space}), see \cite{gur}, which means the
following condition:
\begin{enumerate}
	\item[(G)] \emph{For every isometry $g:X\rightarrow Y$, where $Y$ is a
		finite-dimensional Banach space and $X$ is a subspace of $\mathbb{G}$, and
		every $\varepsilon \in (0,1)$ there exists an $\varepsilon $-isometry $%
		f:Y\rightarrow \mathbb{G}$ such that $f(g(x))=x$ for all $x\in X$.}
\end{enumerate}
A linear operator $f:E\rightarrow F$ between Banach spaces $E$ and $F$ is an 
$\varepsilon $\emph{-isometry} if for $x\in E$ with $\Vert x\Vert =1$ one
has $(1+\varepsilon )^{-1}<\Vert f(x)\Vert <1+\varepsilon $. By an \emph{%
isometry} we mean a linear operator $f:E\rightarrow F$ that is an $%
\varepsilon $-isometry for every $\varepsilon >0$, that is, $\Vert f(x)\Vert
=\Vert x\Vert $ for each $x\in E$.

One can prove easily that the Gurari\v{i} space $\mathbb{G}$ is unique up to
isomorphism of norm arbitrarily close to one. Nevertheless, the question
whether the Gurari\v{i} space is unique up to isometry remained open for a
longer time. It was answered affirmatively by Lusky in 1976, see \cite{lusky}%
, who used quite technical and difficult methods involving techniques
developed by Lazar and Lindenstrauss \cite{lazar}. Much simpler proof has
been provided by Kubi\'s and Solecki in 2013, see \cite{KuSo}.

In \cite{KuSo} the authors proved the following

\begin{theorem}
\label{kuso} Let $E,F$ be separable Gurari\v{i} spaces and $\varepsilon >0$.
Assume $X\subset E$ is a finite-dimensional space and $f:X\rightarrow F$ is
an $\varepsilon -$isometry. Then there exists a bijective isometry $%
h:E\rightarrow F$ such that $\Vert h|_{X}-f\Vert <\varepsilon $.
\end{theorem}

Applying this result for $X$ being the trivial space one gets the result of
Lusky \cite{lusky} stating that the Gurari\v{i} space is unique up to
isometry.

A Banach space $E$ is said to be of \emph{universal disposition} for the
class $\mathcal{U}$ of finite-dimensional spaces if it satisfies the
following condition:
\begin{enumerate}
	\item[(G1)] \emph{For every isometry $j:X\rightarrow Y$, where $Y\in \mathcal{U}$
		and $X\subset E$, there is an isometry $f:Y\rightarrow E$ such that $f\left(
		j\left( x\right) \right) =x$ for all $x\in X$.}		
\end{enumerate}
We refer the reader to \cite{avil} and \cite{garb-kubis}, where recent
developments in the study of Gurari\v{i} spaces, spaces of universal
disposition, and related topics are surveyed.

In the present paper we study non-archimedean counterparts of the above
concepts. The property of being of (almost) universal disposition for
finite-dimensional non-archime\-dean normed spaces is defined precisely in the same way as for the real case mentioned above.

From now on, by $\mathbb{K}$\ we will denote a non-archimedean complete
non-trivially valued field, i.e. the valuation satisfies\emph{\ the strong
triangle inequality}: 
\begin{equation*}
\left\vert \lambda +\mu \right\vert \leq \max \left\{ \left\vert \lambda
\right\vert ,\left\vert \mu \right\vert \right\}
\end{equation*}
for all $\lambda ,\mu \in \mathbb{K}$.

All linear spaces considered in this paper are over $\mathbb{K}$. Recall
that 
\begin{equation*}
\left\vert \mathbb{K}^{\ast }\right\vert =\left\{ \left\vert \lambda
\right\vert :\lambda \in \mathbb{K}\backslash \left\{ 0\right\} \right\}
\end{equation*}
is the \emph{value group} of $\mathbb{K}$.

$\mathbb{K}$ is said to be \emph{discretely valued} if $0$ is the only
accumulation point of $\left\vert \mathbb{K}^{\ast }\right\vert $; then,
there exists a \emph{uniformizing element }$\rho \in \mathbb{K}$ with $%
0<\left\vert \rho \right\vert <1$ such that $\left\vert \mathbb{K}^{\ast
}\right\vert =\left\{ \left\vert \rho \right\vert ^{n}:n\in \mathbb{Z}%
\right\} $. Otherwise, we say that $\mathbb{K}$ is \emph{densely valued}
(then, $\left\vert \mathbb{K}^{\ast }\right\vert $ is a dense subset of $%
[0,\infty )$).

By a \emph{non-archimedean Banach space} we mean a Banach space $E$ equipped
with a non-archimedean norm $\Vert .\Vert $, i.e. a norm for which the
triangle inequality is replaced by a stronger condition $\Vert x+y\Vert \leq
\max \{\Vert x\Vert ,\Vert y\Vert \}$ for all $x,y\in E$.

An infinite-dimensional normed space $E$ over $\mathbb{K}$\ is of \emph{countable type} if it
contains a countable set whose linear hull is dense in $E$. If $\mathbb{K}$
is separable, then a normed space is of countable type if and only if it is
separable.

We say that $E$ (in particular, $E$ may be equal to $\mathbb{K}$) is \emph{spherically complete }%
if every shrinking sequence of balls in $E$ has a non-empty intersection;
otherwise, $E$ is \emph{non-spherically complete}. Every finite-dimensional
Banach space over $\mathbb{K}$ has an equivalent non-archimedean norm. We
refer the reader to the monographs \cite{sc1} and \cite{Rooij} for
non-archimedean concepts mentioned above.

We say that a spherically complete Banach space $\widehat{E}$ is the \emph{%
spherical completion} of a non-archimedean Banach space $E$, if there exists
an isometric embedding $i:E\rightarrow \widehat{E}$ and $\widehat{E}$ has no
proper spherically complete linear subspace containing $i\left( E\right) $.
Applying the natural identification, we will usually identify $E$ with $%
i\left( E\right) .$ Every Banach space (in particular $\mathbb{K}$) has the
spherical completion and any two spherical completions of $E$ are
isometrically isomorphic (\cite[Theorem 4.43]{Rooij}).

Let $\mathcal{U}_{FNA}$ be the class of all non-archimedean
finite-dimensional normed spaces. As it can be expected, properties of
spaces $E$ of (almost) universal disposition for the class $\mathcal{U}%
_{FNA} $ strictly depend on the valued field $\mathbb{K}$, in particular, on
whether it is spherically complete or not. In Section \ref{sect-aud} we show
that all non-archimedean Banach spaces of almost universal disposition for
the class $\mathcal{U}_{FNA}$ are $\varepsilon$-isometric with arbitrarily small $\varepsilon>0$ (Corollary \ref%
{cor-1}). Moreover, all non-archimedean Banach spaces of almost universal disposition
for the class $\mathcal{U}_{FNA}$ are isometrically isomorphic if and only
if $\mathbb{K}$ is spherically complete and $\left\vert \mathbb{K}^{\ast
}\right\vert =\left( 0,\infty \right) $ (Proposition \ref{izo}). The main
result of this section is the following

\begin{theorem}
\label{T-char}
Let $\mathbb{K}$ be a non-archimedean valued field.
The following conditions are equivalent:
\begin{enumerate}
\item[({\rm a})] Every non-archimedean Banach space of countable type over $%
\mathbb{K}$ is of almost universal disposition for the class $\mathcal{U}_{FNA}$.
\item[({\rm b})] $\mathbb{K}$ is densely valued.

\end{enumerate}
\end{theorem}

Section \ref{sect-ud}\ focuses on the study Banach spaces of universal
disposition for the class $\mathcal{U}_{FNA}$. It turns out that a real
Banach space $G$ of almost universal disposition for the class $\mathcal{U}$
can be characterized (see \cite{Kub-game}) by the following
condition:
\begin{enumerate}
	\item[(H)] for every $\varepsilon >0$, for every finite-dimensional normed spaces $%
	X\subset Y$ and for every isometric embedding $j:X\rightarrow G$, $\ $there
	is an isometric embedding $f:Y\rightarrow G$ such that $\left\vert
	\left\vert j-f|_{X}\right\vert \right\vert <\varepsilon$.
\end{enumerate}
In contrast to the real case, in a non-archimedean setting the condition (H)
characterizes Banach spaces of universal disposition for the class $\mathcal{%
U}_{FNA}$. We show the following

\begin{theorem}
\label{Pro-iso}Let $G$ be a non-archimedean Banach space which satisfies the
following property: For every finite-dimensional non-archimedean normed
space $Y$ and every isometric embedding $j:X\rightarrow G$, $\ $where $%
X\subset Y$ is a linear subspace, there is an isometry $f:Y\rightarrow G$
such that $\left\vert \left\vert j-f|_{X}\right\vert \right\vert <1$. Then $%
G $ is of universal disposition for the class $\mathcal{U}_{FNA}$.
\end{theorem}

Recall that a Banach space $X$ is (\emph{isometrically}) \emph{universal}
for the class of Banach spaces $\mathcal{C}$ if $X\in \mathcal{C}$ and for
any $Y\in \mathcal{C}$, there is an isometrical embedding $Y\rightarrow X$.
Note that, as a result of Banach-Mazur theorem, the space \ $C[0,1]$ is
isometrically\emph{\ }universal for the class of separable real Banach
spaces.

If $\mathbb{K}$ is spherically complete, we can properly select a set $I$
and a map $s:I\rightarrow \left( 0,\infty \right) $ such that every
non-archimedean Banach space of countable type can be isometrically embedded
into $E_{u}=c_{0}\left( I:s\right) $. However $E_{u}$ is isometrically
universal for the class of non-archimedean Banach spaces of countable type
if and only if $\mathbb{K}$ is spherically complete and $\left( 0,\infty
\right) $ is an union of at most countably many cosets of $\left\vert 
\mathbb{K}^{\ast }\right\vert $ (Proposition \ref{P-univers}). On the other
hand, $E_{u}$ is never separable. If $\mathbb{K}$ is non-spherically
complete, the role of $c_{0}\left( I:s\right) $ is replaced by $\ell^{\infty }$%
, which clearly is not of countable type (Remark \ref{rem-univ}).

Applying Theorem \ref{Pro-iso} we prove that the spherical completion $%
\widehat{E_{u}}$ of $E_{u}$ is a space of universal disposition for the
class $\mathcal{U}_{FNA}$, see Theorem \ref{Prop-ud}. We show also that the
suitably selected proper linear subspace of $\widehat{E_{u}}$, denoted as $%
E_{h}$, is also of universal disposition for the class $\mathcal{U}_{FNA}$
(Theorem \ref{Eh}). If $\mathbb{K}$ is spherically complete, then $E_{h}=E_{u}$; hence, $E_{h}$ has an orthogonal base and is of countable type
if and only if $\mathbb{K}$ is spherically complete and $\left( 0,\infty
\right) $ is the union of at most countably many cosets of $\left\vert 
\mathbb{K}^{\ast }\right\vert $ (Corollary \ref{Ehh}).

\section{Preliminaries}

Let $t\in (0,1]$. A subset $\left\{ x_{i}:i\in I\right\} \subset E$ is
called $t-$\emph{orthogonal }(\emph{orthogonal \ }for $t=1$) if for each
finite subset $J\subset I$ and all $\left\{ \lambda _{i}\right\} _{i\in
J}\subset \mathbb{K}$ we have 
\begin{equation*}
\left\vert \left\vert \sum_{i\in J}\lambda _{i}x_{i}\right\vert \right\vert
\geq t\cdot \max_{i\in J}\left\vert \left\vert \lambda _{i}x_{i}\right\vert
\right\vert .
\end{equation*}%
If additionally $\overline{\left[ \left\{ x_{i}\right\} _{i\in I}\right] }=E 
$,$\ $then $\left\{ x_{i}\right\} _{i\in I}$ is said to be a $t-$\emph{%
orthogonal base} of $E$. Then every $x\in E$ has an unequivocal expansion 
\begin{equation*}
x=\sum_{i\in I}\lambda _{i}x_{i}\text{ }(\lambda _{i}\in \mathbb{K},\text{ }%
i\in I).
\end{equation*}
Every non-archimedean Banach space of countable type has a $t-$orthogonal
base for each $t\in \left( 0,1\right) $; if $\mathbb{K}$ is spherically
complete, then every non-archimedean Banach space of countable type has an
orthogonal base (\cite[Lemma 5.5]{Rooij}). Every closed linear subspace of a
non-archimedean Banach space with an orthogonal base has an orthogonal base (%
\cite[Theorem 5.9]{Rooij}).

Linear subspaces $D,$ $D_{0}$ of a non-archimedean Banach space $E$ are
called \emph{orthogonal } if $\left\vert \left\vert x+y\right\vert
\right\vert =\max \left\{ \left\vert \left\vert x\right\vert \right\vert
,\left\vert \left\vert y\right\vert \right\vert \right\} $ for all $x\in D$
and $y\in D_{0}$; then we will write $D\perp D_{0}$.

Let $D$ be a closed linear subspace of $E$. Then $D$ is \emph{%
orthocomplemented }in $E$ if there is a linear subspace $\ D_{0}$ \ of $\ E$
\ such that \ $D+D_{0}=E$ \ and $D$ $\perp D_{0}$. Consequently, there
exists a surjective projection (called an \emph{orthoprojection}) $%
P:E\rightarrow D$ with $\left\vert \left\vert P\right\vert \right\vert \leq 1
$. Observe that $D_{1}\perp D_{2}$ implies $D_{1}\cap D_{2}=\varnothing $;
hence the sum $D_{1}+D_{2}$ is direct.

Let $D$ and $E_{0}$ be linear subspaces of a normed space $E$. Recall that $%
E_{0}$ is called an \emph{immediate extension} of $D$ if $D\subset E_{0}$
and there is no nonzero element of $E_{0}$ that is orthogonal to $D$; in
other words, for every $x\in E_{0}\backslash D$ we have $dist\left(
x,D\right) <\left\vert \left\vert x-d\right\vert \right\vert $ for all $d\in
D$. A spherical completion $\widehat{E}$ of $E$ is a maximal immediate
extension of $E$. Let $I$ be a non-empty set and let $s:I\rightarrow \left(
0,\infty \right) $ be a map. By 
\begin{equation*}
\ell^{\infty }\left( I:s\right) :=\{\left( \lambda ^{i}\right) _{i\in I}\in 
\mathbb{K}^{I}:\sup_{i\in I}\left\vert \lambda _{i}\right\vert \cdot s\left(
i\right) <\infty \}
\end{equation*}%
we denote the linear space over $\mathbb{K}$ equipped with the norm 
\begin{equation*}
\left\vert \left\vert \left( \lambda _{i}\right) _{i\in I}\right\vert
\right\vert :=\sup_{i\in I}\left\vert \lambda _{i}\right\vert \cdot s\left(
i\right) .
\end{equation*}%
Then $\ell^{\infty }\left( I:s\right) $ is a non-archimedean Banach space.

Let$\ c_{0}\left( I:s\right) $ be a closed linear subspace of $\ell^{\infty
}\left( I:s\right) $ which consists of all $\left( \lambda ^{i}\right)
_{i\in I}\in \ell^{\infty }\left( I:s\right) $ such that for every $\varepsilon
>0$ there exists a finite $J\subset I$ for which $\left\vert \lambda
^{i}\right\vert \cdot s\left( i\right) <\varepsilon $ for every $i\in
I\backslash J.$ If $s\left( i\right) =1$ for all $i\in I$ we will write $%
\ell^{\infty }\left( I\right) $ and $c_{0}\left( I\right) $, respectively. In
particular $\ell^{\infty }:=\ell^{\infty }\left( \mathbb{N}\right) $ and $%
c_{0}:=c_{0}\left( \mathbb{N}\right) $.

Every\emph{\ }non-archimedean Banach space which has an orthogonal base is
isomorphic with $c_{0}\left( I\right) $ for some set $I$ (see \cite[Ch. 5]%
{Rooij}).

\section{Non-archimedean Banach space of almost universal disposition for
finite-dimensional spaces\label{sect-aud}}

First we prove the following technical fact.

\begin{lemma}
\label{Lem-1}Let $E$ be a non-archimedean Banach space of countable type,
let $F$ be a finite-dimensional linear subspace of $E$, $t\in (0,1)$ and $%
\left\{ x_{1},...,x_{m}\right\} $ be a $\sqrt{t}-$orthogonal base of $F$.
Then there exist $x_{m+1},x_{m+2},...\in E\backslash F$ such that $\left(
x_{n}\right) $ is a $t-$orthogonal base of $E$.
\end{lemma}

\begin{proof}
By \cite[Theorem 2.3.13]{sc1} there exists a linear subspace $F_{0}\subset E$
such that $E=F\oplus F_{0}$ and 
\begin{equation*}
\left\vert \left\vert u_{1}+u_{2}\right\vert \right\vert \geq \sqrt{t}\cdot
\max \left\{ \left\vert \left\vert u_{1}\right\vert \right\vert ,\left\vert
\left\vert u_{2}\right\vert \right\vert \right\}
\end{equation*}
for all $u_{1}\in F,u_{2}\in F_{0}$. Applying \cite[Theorem 2.3.7]{sc1} we
select a $\sqrt{t}-$orthogonal base $\left( z_{n}\right) $ of $F_{0}$.
Denote $x_{m+n}:=z_{n}$ for every $n\in \mathbb{N}$. Then taking any $k\in 
\mathbb{N}$ and $\lambda _{1},...,\lambda _{k}\in \mathbb{K}$ one gets

\begin{multline*}
\left\vert \left\vert \sum_{i=1}^{k}\lambda _{i}x_{i}\right\vert \right\vert
\geq \sqrt{t}\cdot \max \left\{ \left\vert \left\vert \sum_{i=1}^{m}\lambda
_{i}x_{i}\right\vert \right\vert ,\left\vert \left\vert
\sum_{i=m+1}^{k}\lambda _{i}x_{i}\right\vert \right\vert \right\} \\
\geq \sqrt{t}\cdot \max \left\{ \sqrt{t}\cdot \max_{i=1,...,m}\left\vert
\left\vert \lambda _{i}x_{i}\right\vert \right\vert ,\sqrt{t}\cdot
\max_{i=m+1,...,k}\left\vert \left\vert \lambda _{i}x_{i}\right\vert
\right\vert \right\} \geq t\cdot \max_{i=1,...,k}\left\vert \left\vert
\lambda _{i}x_{i}\right\vert \right\vert .
\end{multline*}

Hence, $\left( x_{n}\right) $ is a required $t-$orthogonal base of $E$.
\end{proof}

\begin{theorem}
\label{Th-aud}A non-archimedean Banach space of countable type $E$ is of
almost universal disposition for the class $\mathcal{U}_{FNA}$ if and only
if $\left\vert \left\vert E^{\ast }\right\vert \right\vert $ is a dense
subset of $\left( 0,\infty \right) $.
\end{theorem}

\begin{proof}
Let $\ E$ be a non-archimedean Banach space of countable type. Assume that $%
\left\vert \left\vert E^{\ast }\right\vert \right\vert $ is a dense subset
of $\left( 0,\infty \right) $. Let $X$ be a finite-dimensional subspace of $%
E$, $Y$ be a non-archimedean finite-dimensional normed space and $%
i:X\rightarrow Y$ be an isometrical embedding. Assume that $\dim Y=m$ and $\dim
X=m_{0}$; clearly, $m_{0}\leq m$. Fix $\varepsilon >0\ $and take $t\in (%
\frac{1}{\sqrt[3]{1+\varepsilon }},1)$. Applying Lemma \ref{Lem-1} and \cite[%
Theorem 2.3.7]{sc1} we form a $t-$orthogonal base $\left( x_{n}\right) $ of $%
E$ such that $\left\{ x_{1},...,x_{m_{0}}\right\} $ is a $\sqrt{t}-$%
orthogonal base of $X$. Now, applying Lemma \ref{Lem-1} again, we select $%
y_{m_{0}+1},...,y_{m}\in Y$ such that $\left\{ y_{1},...,y_{m}\right\} $ is
a $t-$orthogonal base of $Y$, where $y_{k}=i\left( x_{k}\right) $ for $k\in
\{1,...,m_{0}\}$. Since, by assumption $\left\vert \left\vert E^{\ast
}\right\vert \right\vert $ is a dense subset of $\left( 0,\infty \right) $,
we can assume that $1\geq \left\vert \left\vert y_{k}\right\vert \right\vert
\geq t$ $\left( k=1,...,m\right) $ and $1\geq \left\vert \left\vert
x_{n}\right\vert \right\vert \geq t$ $\left( n\in \mathbb{N}\right) $; thus, 
$\left\vert \left\vert y_{k}\right\vert \right\vert \geq t\cdot \left\vert
\left\vert x_{k}\right\vert \right\vert $ for each $k\in \left\{
1,...,m\right\} $.

Define $f:Y\rightarrow E$ by setting%
\begin{equation*}
f\left( \sum_{k=1}^{m}\lambda _{k}y_{k}\right) =\sum_{k=1}^{m}\lambda
_{k}x_{k}.
\end{equation*}%
Clearly, $f\left( i\left( x\right) \right) =x$ for all $x\in X$. Let $%
y=\sum_{k=1}^{m}\lambda _{k}y_{k}\in Y$. Then, 
\begin{equation*}
\left\vert \left\vert y\right\vert \right\vert =\left\vert \left\vert
\sum_{k=1}^{m}\lambda _{k}y_{k}\right\vert \right\vert \geq t\cdot
\max_{k=1,...,m}\{\left\vert |\lambda _{k}y_{k}|\right\vert \}\geq
t^{2}\cdot \max_{k=1,...,m}\{\left\vert |\lambda _{k}x_{k}|\right\vert
\}\geq t^{3}\cdot \left\vert \left\vert \sum_{k=1}^{m}\lambda
_{k}x_{k}\right\vert \right\vert =t^{3}\cdot \left\vert \left\vert f\left(
y\right) \right\vert \right\vert .
\end{equation*}%
On the other hand, we have

\begin{equation*}
\left\vert \left\vert y\right\vert \right\vert =\left\vert \left\vert
\sum_{k=1}^{m}\lambda _{k}y_{k}\right\vert \right\vert \leq
\max_{k=1,...,m}\{|\left\vert |\lambda _{k}y_{k}|\right\vert \}\leq \frac{1}{%
t}\max_{k=1,...,m}\{\left\vert |\lambda _{k}x_{k}|\right\vert \}\leq \frac{1%
}{t^{2}}\left\vert \left\vert \sum_{k=1}^{m}\lambda _{k}x_{k}\right\vert
\right\vert =\frac{1}{t^{2}}\cdot \left\vert \left\vert f\left( y\right)
\right\vert \right\vert .
\end{equation*}%
Thus $\left( 1-\varepsilon \right) \left\vert \left\vert y\right\vert
\right\vert \leq \left\vert \left\vert f\left( y\right) \right\vert
\right\vert \leq \left( 1+\varepsilon \right) \left\vert \left\vert
y\right\vert \right\vert $.

Now assume that $\left\vert \left\vert E^{\ast }\right\vert \right\vert $ is
not dense in $\left( 0,\infty \right) $. Then there exist $s_{1}\in \left(
0,\infty \right) $ and $\varepsilon >0$ such that 
\begin{equation*}
\left\vert \left\vert E^{\ast }\right\vert \right\vert \cap \left(
s_{1}-2\varepsilon \cdot s_{1},s_{1}+2\varepsilon \cdot s_{1}\right)
=\varnothing .
\end{equation*}

Define $X=\mathbb{K}$ and $Y=\left( \mathbb{K}^{2},\left\vert \left\vert
.\right\vert \right\vert _{Y}\right) ,$ where $\left\vert \left\vert \left(
\lambda _{1},\lambda _{2}\right) \right\vert \right\vert _{Y}:=\max \left\{
\left\vert \lambda _{1}\right\vert ,s_{1}\cdot \left\vert \lambda
_{2}\right\vert \right\} ,$ $\lambda _{1},\lambda _{2}\in \mathbb{K}$.

Assume for a contradiction that there is an $\varepsilon -$isometry $%
f:Y\rightarrow E$. But then, taking $x_{0}=(0,1)\in Y$, we obtain $%
\left\vert \left\vert x_{0}\right\vert \right\vert _{Y}=s_{1}$ and $%
\left\vert \left\vert f\left( x_{0}\right) \right\vert \right\vert \geq
s_{1}+2\varepsilon \cdot s_{1}$ or $\left\vert \left\vert f\left(
x_{0}\right) \right\vert \right\vert \leq s_{1}-2\varepsilon \cdot s_{1}$.
Hence, $\left\vert \left\vert f\left( x_{0}\right) \right\vert \right\vert
>\left( 1+\varepsilon \right) \left\vert \left\vert x_{0}\right\vert
\right\vert $, or $\left\vert \left\vert f\left( x_{0}\right) \right\vert
\right\vert <\left( 1-\varepsilon \right) \left\vert \left\vert
x_{0}\right\vert \right\vert $, a contradiction.
\end{proof}

Next conclusion follows directly from Theorem \ref{Th-aud}.

\begin{corollary}
\label{cor-1}If $\mathbb{K}$ is densely valued, every non-archimedean Banach
space of countable type is of almost universal disposition for the class $%
\mathcal{U}_{FNA}$. All non-archimedean Banach spaces of almost universal
disposition for the class $\mathcal{U}_{FNA}$ are $\varepsilon$-isometric, where $\varepsilon>0$ is arbitrarily small.
\end{corollary}

However, as the next result shows, they need not be always isometric.

\begin{proposition}
\label{izo} All non-archimedean Banach spaces over $\mathbb{K}$ of almost
universal disposition for the class $\mathcal{U}_{FNA}$ are isometrically
isomorphic if and only if $\mathbb{K}$ is spherically complete and $%
\left\vert \mathbb{K}^{\ast }\right\vert =\left( 0,\infty \right) $.
\end{proposition}

\begin{proof}
If $\mathbb{K}$ is spherically complete and $\left\vert \mathbb{K}^{\ast
}\right\vert =\left( 0,\infty \right) $ then every non-archimedean Banach
space of countable type has an orthonormal base, thus, it is isometrically
isomorphic with $c_{0}$. Hence, the conclusion follows.

Now assume that $\mathbb{K}$ is non-spherically complete. Then, since $%
\mathbb{K}$ is densely valued, $c_{0}$ and $\mathbb{K}_{v}^{2}\oplus c_{0}$
are both of universal disposition for the class $\mathcal{U}_{FNA}$ by
Theorem \ref{Th-aud} (recall that $\mathbb{K}_{v}^{2}$ is a two-dimensional
normed space without two orthogonal elements, see \cite[Example 2.3.26]{sc1}%
). Clearly, $\mathbb{K}_{v}^{2}\oplus c_{0}$ and $c_{0}$ are not
isometrically isomorphic.

Suppose that $\left\vert \mathbb{K}^{\ast }\right\vert \neq \left( 0,\infty
\right) $. Then we can find $s\in \left( 0,\infty \right) \backslash
\left\vert \mathbb{K}^{\ast }\right\vert $. Define the norm $\left\vert
\left\vert x\right\vert \right\vert _{s}:c_{0}\rightarrow \lbrack 0,\infty )$
by 
\begin{equation*}
\left\vert \left\vert x\right\vert \right\vert _{s}:=\max \{s\cdot
\left\vert x_{1}\right\vert ,\max_{n>1}\{\left\vert x_{n}\right\vert \}\},
\end{equation*}%
and $x=\left( x_{n}\right) \in c_{0}$. Then, by Theorem \ref{Th-aud}, $%
E=\left( c_{0},\left\vert \left\vert .\right\vert \right\vert _{s}\right) $
and $F=\left( c_{0},\left\vert \left\vert .\right\vert \right\vert _{\infty
}\right) $ are of almost universal disposition for the class $\mathcal{U}%
_{FNA}$. Since $\left\vert \left\vert E\right\vert \right\vert \neq
\left\vert \left\vert F\right\vert \right\vert ,$ $E$ and $F$ are not
isometrically isomorphic.
\end{proof}

Now, we are ready to prove Theorem \ref{T-char}, which characterizes $%
\mathbb{K}$, formulated in Introduction.

\begin{proof}[Proof of Theorem \protect\ref{T-char}]
Let $E$ be a non-archimedean Banach space of countable type. If $\mathbb{K}$
is densely valued, $\left\vert \left\vert E^{\ast }\right\vert \right\vert $
is a dense subset of $\left( 0,\infty \right) $ and the conclusion follows
from Theorem \ref{Th-aud}. Assume now that $\mathbb{K}$ is discretely valued
and $\rho $ be a uniformizing element of $\mathbb{K}$. Let $E:=c_{0}$. Set $%
s:=\frac{\left\vert \rho \right\vert +1}{2}$ and take 
\begin{equation*}
\varepsilon <\frac{1-s}{s}=\frac{s-\left\vert \rho \right\vert }{s}.
\end{equation*}%
Let $X=[e_{1}]\subset E$ and $Y=\left( \mathbb{K}^{2},\left\vert \left\vert
.\right\vert \right\vert _{s}\right) $, where 
\begin{equation*}
\left\vert \left\vert x\right\vert \right\vert _{s}:=\max \{\left\vert
x_{1}\right\vert ,s\cdot \left\vert x_{2}\right\vert \}\},\left(
x_{1},x_{2}\right) \in \mathbb{K}^{2}.
\end{equation*}
Define an isometry $i:X\rightarrow Y$ by $i(\lambda e_{1}):=\left( \lambda
x_{1},0\right) $ and assume that there exists an $\varepsilon -$isometry $%
f:Y\rightarrow E$. Then, for $x=\left( 0,1\right) \in Y$ we get 
\begin{equation*}
\left( 1-\varepsilon \right) \cdot s\leq \left\vert \left\vert f\left(
x\right) \right\vert \right\vert \leq \left( 1+\varepsilon \right) \cdot s.
\end{equation*}%
Recall that $\left\vert \left\vert E^{\ast }\right\vert \right\vert =\left\{
\left\vert \rho \right\vert ^{n}:n\in \mathbb{Z}\right\} $, hence $\left(
\left\vert \rho \right\vert ,1\right) \cap \left\vert \left\vert E^{\ast
}\right\vert \right\vert =\varnothing $. But $\left( 1-\varepsilon \right)
\cdot s>\left\vert \rho \right\vert $ and $\left( 1+\varepsilon \right) s<1$%
, a contradiction.
\end{proof}

\section{Non-archimedean Banach space of universal disposition for
finite-dimensional spaces\label{sect-ud}$\protect $}

We start with the proof of Theorem \ref{Pro-iso}, as promised in
Introduction.

\begin{proof}[Proof of Theorem \protect\ref{Pro-iso}]
We need to prove that there is an isometry $T:Y\rightarrow G$ which extends $%
j $. Let $t=\left\vert \left\vert j-f|_{X}\right\vert \right\vert ,$ $%
n_{X}=\dim X$ and $n_{Y}=\dim Y$. Clearly $n_{X}\leq n_{Y}$. Choose $\left\{
x_{1},...,x_{n_{Y}}\right\} $, a $\sqrt{t}-$orthogonal base of $Y$ such that 
$\left[ x_{1},...,x_{n_{X}}\right] =X$. Define $T:Y\rightarrow G$ by setting 
\begin{equation*}
T\left( x_{n}\right) :=\left\{ 
\begin{array}{cc}
j(x_{n}) & \text{ if }n\leq n_{X} \\ 
f\left( x_{n}\right) & \text{if }n>n_{X}%
\end{array}%
\right. ,n=1,...,n_{Y}.
\end{equation*}%
Clearly $T$ extends $j$. We show that $T$ is an isometry. Take $x\in Y$,
written as $x=\sum_{i=1}^{n_{Y}}\lambda _{i}x_{i}$ $\left( \lambda _{i}\in 
\mathbb{K}\right) $. Then, we obtain 
\begin{subequations}
\label{p1}
\begin{multline}
\left\vert \left\vert T(x)\right\vert \right\vert =\left\vert \left\vert
\sum_{i=1}^{n_{X}}\lambda _{i}T(x_{i})+\sum_{i=n_{X}+1}^{n_{Y}}\lambda
_{i}T(x_{i})\right\vert \right\vert  \label{p1a} \\
=\left\vert \left\vert \sum_{i=1}^{n_{X}}\lambda
_{i}j(x_{i})-\sum_{i=1}^{n_{X}}\lambda
_{i}f(x_{i})+\sum_{i=1}^{n_{X}}\lambda
_{i}f(x_{i})+\sum_{i=n_{X}+1}^{n_{Y}}\lambda _{i}f(x_{i})\right\vert
\right\vert \\
=\left\vert \left\vert (j-f)(\sum_{i=1}^{n_{X}}\lambda
_{i}x_{i})+f(\sum_{i=1}^{n_{Y}}\lambda _{i}x_{i})\right\vert \right\vert
=\left\vert \left\vert (j-f)(x_{0})+f\left( x\right) \right\vert \right\vert
,  \notag
\end{multline}%
where $x_{0}=\sum_{i=1}^{n_{X}}\lambda _{i}x_{i}\in X$. But, 
\end{subequations}
\begin{equation*}
\left\vert \left\vert x\right\vert \right\vert \geq \sqrt{t}\cdot \max
\left\{ \left\vert \left\vert x_{0}\right\vert \right\vert
,\sum_{i=n_{X}+1}^{n_{Y}}\lambda _{i}x_{i}\right\} \geq \sqrt{t}\cdot
\left\vert \left\vert x_{0}\right\vert \right\vert
\end{equation*}%
and%
\begin{equation*}
\left\vert \left\vert (j-f)(x_{0})\right\vert \right\vert \leq \left\vert
\left\vert j-f|_{X}\right\vert \right\vert \cdot \left\vert \left\vert
x_{0}\right\vert \right\vert =t\cdot \left\vert \left\vert x_{0}\right\vert
\right\vert \leq \sqrt{t}\cdot \left\vert \left\vert x\right\vert
\right\vert <\left\vert \left\vert x\right\vert \right\vert =\left\vert
\left\vert f\left( x\right) \right\vert \right\vert
\end{equation*}%
hence, we get%
\begin{equation*}
\left\vert \left\vert T(x)\right\vert \right\vert =\left\vert \left\vert
(j-f)(x_{0})+f\left( x\right) \right\vert \right\vert =\left\vert \left\vert
f\left( x\right) \right\vert \right\vert =\left\vert \left\vert x\right\vert
\right\vert .
\end{equation*}
\end{proof}

To prove the main result of this section we need a few technical lemmas (see
also \cite[Exercise 5.B and Lemma 4.42]{Rooij} and \cite[Theorem 2.3.16]{sc1}%
).

\begin{lemma}
\label{L-ort}Let $0<t\leq 1$ and let $\left\{ x_{1},...,x_{n}\right\} $ be a 
$t-$orthogonal set in a non-archimedean normed space $E$. If $\left\{
z_{1},...,z_{n}\right\} \subset E$ and $\left\vert \left\vert
x_{i}-z_{i}\right\vert \right\vert <t\cdot \left\vert \left\vert
x_{i}\right\vert \right\vert $ for each $i\in \left\{ 1,...,n\right\} $,
then $\left\{ z_{1},...,z_{n}\right\} $ is also $t-$orthogonal.
\end{lemma}

\begin{proof}
Take any $\lambda _{i}\in \mathbb{K},$ $i=1,...,n$. Since $\left\vert
\left\vert x_{i}-z_{i}\right\vert \right\vert <t\cdot \left\vert \left\vert
x_{i}\right\vert \right\vert $, we have $\left\vert \left\vert
z_{i}\right\vert \right\vert =\left\vert \left\vert x_{i}\right\vert
\right\vert $ for each $i\in \left\{ 1,...,n\right\} .$ Consequently we note 
\begin{equation*}
\left\vert \left\vert \lambda _{1}x_{1}+...+\lambda _{n}x_{n}\right\vert
\right\vert \geq t\cdot \max_{i=1,...,n}\left\vert \left\vert \lambda
_{i}x_{i}\right\vert \right\vert
\end{equation*}
and

\begin{equation*}
\left\vert \left\vert \lambda _{1}(z_{1}-x_{1})+...+\lambda
_{n}(z_{n}-x_{n})\right\vert \right\vert \leq \max_{i=1,...,n}\left\vert
\left\vert \lambda _{i}(z_{i}-x_{i})\right\vert \right\vert <t\cdot
\max_{i=1,...,n}\left\vert \left\vert \lambda _{i}x_{i}\right\vert
\right\vert .
\end{equation*}%
Hence%
\begin{multline*}
\left\vert \left\vert \lambda _{1}z_{1}+...+\lambda _{n}z_{n}\right\vert
\right\vert =\left\vert \left\vert \lambda _{1}z_{1}+...+\lambda
_{n}z_{n}-\left( \lambda _{1}x_{1}+...+\lambda _{n}x_{n}\right) +\left(
\lambda _{1}x_{1}+...+\lambda _{n}x_{n}\right) \right\vert \right\vert \\
=\left\vert \left\vert \lambda _{1}(z_{1}-x_{1})+...+\lambda
_{n}(z_{n}-x_{n})+\left( \lambda _{1}x_{1}+...+\lambda _{n}x_{n}\right)
\right\vert \right\vert \\
=\left\vert \left\vert \lambda _{1}x_{1}+...+\lambda _{n}x_{n}\right\vert
\right\vert \geq t\cdot \max_{i=1,...,n}\left\vert \left\vert \lambda
_{i}x_{i}\right\vert \right\vert =t\cdot \max_{i=1,...,n}\left\vert
\left\vert \lambda _{i}z_{i}\right\vert \right\vert,
\end{multline*}%
and we are done.
\end{proof}

\begin{lemma}
\label{ext-isom}Let $E,F$ be non-archimedean normed spaces, $D$ a linear
subspace of $E$ such that $E$ is an immediate extension of $D$, let $F$ be
spherically complete and $T:D\rightarrow F$ be an isometry. Then $T$ can be
extended to a linear isometry $T^{\prime }:E\rightarrow F$.
\end{lemma}

\begin{proof}
Applying Ingleton's theorem (see \cite[Theorem 4.8]{Rooij}) we can extend
the isometry $T:D\rightarrow F$ to the linear operator $T^{\prime
}:E\rightarrow F$ such that $\left\vert \left\vert T^{\prime }\right\vert
\right\vert \leq 1$. We prove that $T^{\prime }$ is also an isometry.

Set $x\in E\backslash D.$ Then, since $E$ is an immediate extension of $D$,
there is $x_{d}\in D$ such that 
\begin{equation*}
\left\vert \left\vert x-x_{d}\right\vert \right\vert <\left\vert \left\vert
x_{d}\right\vert \right\vert =\left\vert \left\vert x\right\vert \right\vert.
\end{equation*}
Thus we get%
\begin{equation*}
\left\vert \left\vert T^{\prime }\left( x\right) -T^{\prime }\left(
x_{d}\right) \right\vert \right\vert \leq \left\vert \left\vert T^{\prime
}\right\vert \right\vert \cdot \left\vert \left\vert x-x_{d}\right\vert
\right\vert <\left\vert \left\vert x_{d}\right\vert \right\vert =\left\vert
\left\vert T^{\prime }\left( x_{d}\right) \right\vert \right\vert
\end{equation*}%
and%
\begin{equation*}
\left\vert \left\vert T^{\prime }\left( x\right) \right\vert \right\vert
=\left\vert \left\vert T^{\prime }\left( x\right) -T^{\prime }\left(
x_{d}\right) +T^{\prime }\left( x_{d}\right) \right\vert \right\vert
=\left\vert \left\vert T^{\prime }\left( x_{d}\right) \right\vert
\right\vert .
\end{equation*}%
Hence, finally $\left\vert \left\vert T^{\prime }\left( x\right) \right\vert
\right\vert =\left\vert \left\vert x\right\vert \right\vert $.
\end{proof}

\begin{lemma}
\label{nowy}Let $Y$ be a finite-dimensional non-archimedean normed space and 
$X$ be its linear subspace. Let $\left\{ u_{1},...u_{m_{Y}}\right\} $ be a
maximal orthogonal set in $Y$ such that $\left\{ u_{1},...,u_{m_{X}}\right\} 
$ is a maximal orthogonal set in $X$ for some $m_{X}\leq m_{Y}$ and let $F_{Y}:=%
\left[ u_{m_{X}+1},...,u_{m_{Y}}\right] $. Then, $F_{Y}\perp X$.
\end{lemma}

\begin{proof}
Assume that $m_{Y}>m_{X}$ (otherwise nothing is to prove). Take any $x\in X$
and $y\in F_{Y}$. If $x\in \left[ u_{1},...,u_{m_{X}}\right] $, the
conclusion is obvious. So, assume that $x\notin \left[ u_{1},...,u_{m_{X}}%
\right] $. But then, since $X$ is an immediate extension of $\left[
u_{1},...,u_{m_{X}}\right] $, there is $x_{0}\in \left[ u_{1},...,u_{m_{X}}%
\right] $ with 
\begin{equation*}
\left\vert \left\vert x-x_{0}\right\vert \right\vert <\left\vert \left\vert
x\right\vert \right\vert =\left\vert \left\vert x_{0}\right\vert \right\vert
.
\end{equation*}%
Thus, since 
\begin{equation*}
\left\vert \left\vert x_{0}+y\right\vert \right\vert =\max \left\{
\left\vert \left\vert x_{0}\right\vert \right\vert ,\left\vert \left\vert
y\right\vert \right\vert \right\} ,
\end{equation*}%
we have 
\begin{equation*}
\left\vert \left\vert x+y\right\vert \right\vert =\left\vert \left\vert
x-x_{0}+x_{0}+y\right\vert \right\vert =\left\vert \left\vert
x_{0}+y\right\vert \right\vert =\max \left\{ \left\vert \left\vert
x_{0}\right\vert \right\vert ,\left\vert \left\vert y\right\vert \right\vert
\right\} =\max \left\{ \left\vert \left\vert x\right\vert \right\vert
,\left\vert \left\vert y\right\vert \right\vert \right\} .
\end{equation*}
\end{proof}

Let $r:=\left\vert \rho \right\vert $ if $\mathbb{K}$ is discretely valued,
where $\rho \in \mathbb{K}$ is a uniformizing element of $\left\vert 
\mathbb{K}^{\ast }\right\vert $ with $0<\left\vert \rho \right\vert <1$, and
let $r$ be any number taken from $\left( \frac{1}{2},1\right) $ if $\mathbb{K%
}$ is densely valued. Note that $\left( 0,\infty \right) $ is a
multiplicative group. Let 
\begin{equation}
\pi :\left( 0,\infty \right) \rightarrow G:=\left( 0,\infty \right)
/\left\vert \mathbb{K}^{\ast }\right\vert  \label{map-pi}
\end{equation}%
be the quotient map and let $S=\left\{ s_{g}:g\in G\right\} $ be a set of
representatives of elements of $G$ in $(r,1],$ i.e. $\pi \left( s_{g}\right)
=g$.

Let $I_{u}$ be a set for which $\card\left( I_{u}\right) =\max \left\{ \aleph
_{0},\card(G)\right\} $ and let $I_{u}=\bigcup_{g\in G}I_{g}$ where $\left\{
I_{g}:g\in G\right\} $ is a partition of $I_{u}$ such that $\card\left(
I_{g}\right) =\aleph _{0}$ for each $g\in G$. Then, clearly $c_{0}\left(
I_{u}\right) =\bigoplus_{g\in G}c_{0}\left( I_{g}\right) $.

Define the function $s:I_{u}\mathbb{\rightarrow }(r,1]$ by $h\left( i\right)
:=s_{g}$ if $i\in I_{g}$ and the norm on $c_{0}\left( I_{u}\right) $ by 
\begin{equation*}
\left\vert \left\vert x\right\vert \right\vert _{u}:=\max_{i\in
I_{u}}\left\{ s\left( i\right) \cdot \left\vert x_{i}\right\vert \right\} ,%
\text{ }x=\left( x_{i}\right) _{i\in I}\in c_{0}\left( I_{u}\right) .
\end{equation*}%
Denote $E_{u}:=\left( c_{0}\left( I_{u}\right) ,\left\vert \left\vert
.\right\vert \right\vert _{u}\right) $.

\begin{proposition}
\label{P-univers}Let $\mathbb{K}$ be spherically complete. Then every
non-archimedean Banach space of countable type can be isometrically embedded
into $E_{u}$.

\begin{itemize}
\item $E_{u}$ is of countable type (hence $E_{u}$ is isometrically universal
for the class of non-archimedean Banach spaces of countable type) if and
only if $\left( 0,\infty \right) $ is the union of at most countably many
cosets of $\left\vert \mathbb{K}^{\ast }\right\vert $;

\item $E_{u}$ is never separable.
\end{itemize}
\end{proposition}

\begin{proof}
Let $E$ be a non-archimedean Banach space of countable type. Since $\mathbb{K%
}$ is spherically complete, $E$ has an orthogonal base $\left( x_{n}\right) $
(see \cite[Theorem 2.3.25]{sc1}). Let%
\begin{equation*}
J_{g}=\left\{ n:\pi \left( \left\vert \left\vert x_{n}\right\vert
\right\vert \right) =g\right\} ,\text{ \ }g\in G,
\end{equation*}%
where $\pi $ is the map defined in (\ref{map-pi}). Then $G_{0}=\{g\in G:$ \ $%
J_{g}$ is nonempty$\}$ is at most countable. So, we can write $\mathbb{N=}%
\bigcup_{g\in G_{0}}J_{g}$, where $J_{g}$ $\left( g\in G_{0}\right) $ are
nonempty, finite or infinite, pairwise disjoint subsets of $\mathbb{N}$. 

Define the map $l:\mathbb{N\rightarrow }I_{u}$ (recall that $%
I_{u}=\bigcup_{g\in G}I_{g}$ and $I_{g}$ is countable for every $g\in G$;
thus we can write $I_{g}=\left\{ m_{1}^{g},m_{2}^{g},...\right\} $, $g\in G$%
) as follows: for every $n\in \mathbb{N}$ there exist $g\in G_{0}$ and $k\in 
\mathbb{N}$ such that $n=n_{k}^{g}$ (what means $n\in J_{g}$). Finally set $%
l\left( n\right) :=m_{k}^{g}$. 

Note that for every $n\in \mathbb{N}$ we can find $\lambda _{n}\in \mathbb{K}
$ for which 
\begin{equation*}
\left\vert \left\vert x_{n}\right\vert \right\vert =s_{g}\cdot \left\vert
\lambda _{n}\right\vert .
\end{equation*}

Next, define the map $i_{0}:\left\{ x_{1},x_{2},...\right\} \rightarrow E_{u}
$ by the formula $x_{n}\longmapsto \lambda _{n}e_{l\left( n\right) }$, where 
$e_{n}$ are as usual the unit vectors. Since 
\begin{equation*}
\left\vert \left\vert \lambda _{n}e_{l\left( n\right) }\right\vert
\right\vert _{u}=s_{g}\cdot \left\vert \lambda _{n}\right\vert =\left\vert
\left\vert x_{n}\right\vert \right\vert ,
\end{equation*}%
we can extend the map $i_{0}$ to an isometric embedding $E\rightarrow E_{u}$.

Now we prove the next claim of the proposition. Suppose that $\left(
0,\infty \right) $ is not the union of at most countably many cosets of $%
\left\vert \mathbb{K}^{\ast }\right\vert $ and assume that $E_{u}$ is of
countable type. By \cite[Theorem 2.3.25]{sc1} the space $E_{u}$ has an
orthogonal base $\left( x_{n}\right) $. Since $\left( x_{n}\right) $ is
orthogonal, $\left\vert \left\vert E_{u}\right\vert \right\vert
_{u}\backslash \left\{ 0\right\} $ consist of at most countably many cosets
of $\left\vert \mathbb{K}^{\ast }\right\vert $. Hence there exists 
\begin{equation*}
s\in \left( 0,\infty \right) \backslash \left\vert \left\vert
E_{u}\right\vert \right\vert _{u}.
\end{equation*}%
Define $E=\left( \mathbb{K}^{2},\left\vert \left\vert .\right\vert
\right\vert _{s}\right) ,$ where 
\begin{equation*}
\left\vert \left\vert (x,y)\right\vert \right\vert _{s}:=\max \{\left\vert
x\right\vert ,s\cdot \left\vert y\right\vert \},
\end{equation*}%
$(x,y)\in \mathbb{K}^{2}.$ Then 
\begin{equation*}
\left\vert \left\vert (0,1)\right\vert \right\vert _{s}\notin \left\vert
\left\vert E_{u}\right\vert \right\vert _{u}.
\end{equation*}%
Hence there is no isometry $E\rightarrow E_{u}$, a contradiction. If $\left(
0,\infty \right) $ is the union of at most countably many cosets of $%
\left\vert \mathbb{K}^{\ast }\right\vert $, then $G$, and consequently $I_{u}
$ is countable, and $E_{u}$ is of countable type.

Finally, assume that $\mathbb{K}$\ is separable. If $\mathbb{K}$ is
discretely valued, then $\left( 0,\infty \right) $ is always the union of
more than countably many cosets of $\left\vert \mathbb{K}^{\ast }\right\vert
.$ On the other hand, by \cite[Theorem 20.5]{Schikhof-UC} there is no
separable densely valued spherically complete $\mathbb{K};$ hence $E_{u}$ is
not separable.
\end{proof}

\begin{remark}
\label{rem-univ}If $\mathbb{K}$ is non-spherically complete, then $E_{u}$
does not contain an isometric image of any non-archimedean Banach spaces of
countable type. Indeed, in this case there exists finite-dimensional normed
spaces without orthogonal bases, see \cite[Example 2.3.26]{sc1} and \cite%
{AK-Cp}. Take $E=\mathbb{K}_{v}^{2}$, where $\mathbb{K}_{v}^{2}$ is a
two-dimensional normed space over $\mathbb{K}$ without two non-zero
orthogonal elements, and assume that there exists an isometric embedding $%
i:E\rightarrow E_{u}$. Then the image $i\left( E\right) $ has no two
non-zero orthogonal elements. But this contradicts the conclusion of
Gruson's theorem (\cite[Theorem 5.9]{Rooij}) stating that every linear
subspace of a non-archimedean Banach spaces with an orthogonal base has an
orthogonal base.

If $\mathbb{K}$ is non-spherically complete the role of $c_{0}\left(
I:s\right) $ takes the space $\ell^{\infty }$. In this case, by \cite[Theorem
2.5.13]{sc1}, every\emph{\ }non-archimedean Banach space of countable type
can be isometrically embedded into $\ell^{\infty }$.
\end{remark}

Finally we prove the following

\begin{theorem}
\label{Prop-ud} The spherical completion $\widehat{E_{u}}$ of $E_{u}$ is a
non-archimedean Banach space of universal disposition for the class $%
\mathcal{U}_{FNA}$.
\end{theorem}

\begin{proof}
Denote $F:=\widehat{E_{u}}$. Let $X\subset F$ and let $j:X\rightarrow Y$ $\ $%
be an isometric embedding, where $Y$ is a finite non-archimedean normed
space. We prove, that there exists an isometric embedding $f:Y\rightarrow F$
such that $f\left( j\left( x\right) \right) =x$ for all $x\in X$.

Choose a maximal orthogonal set $\left\{
u_{1},...,u_{m_{X}},...,u_{m_{Y}}\right\} $ in $Y$ such that $\left\{
u_{1},...,u_{m_{X}}\right\} $ is a maximal orthogonal set in $j(X)$ for some 
$m_{Y}\geq m_{X}\geq 1$. Set $F_{Y}:=\left[ u_{m_{X}+1},...,u_{m_{Y}}\right] 
$. By Lemma \ref{nowy}\ we get $F_{Y}\perp j(X)$. \ 

Set $v_{k}=f\left( u_{k}\right) :=j^{-1}\left( u_{k}\right) $ for each $k\in
\{1,...,m_{X}\}$. For every $k\in \left\{ m_{X+1},...,m_{Y}\right\} $ choose 
$i_{k}\in I$ \ such that $\left\vert \left\vert e_{i_{k}}\right\vert
\right\vert _{u}=\left\vert \left\vert \lambda _{k}u_{k}\right\vert
\right\vert $ for some $\lambda _{k}\in \mathbb{K}$ and 
\begin{equation*}
e_{i_{k}}\perp \lbrack v_{1},...,v_{m_{X}},e_{i_{m_{X+1}}},...,e_{i_{k-1}}].
\end{equation*}%
Next set $f\left( u_{k}\right) :=e_{i_{k}}$ for $k=m_{X+1},...,m_{Y}$.
Define $f:j\left( X\right) +F_{Y}\rightarrow F$. Clearly $f$ is an isometry
and $f\left( j\left( X\right) +F_{Y}\right) \subset E_{h}^{0}$.

If $\mathbb{K}$ is spherically complete, we are done, as $\left\{
u_{1},...,u_{m_{Y}}\right\} $ is an orthogonal base of $Y$ by \cite[Lemma
5.5 and Theorem 5.15]{Rooij}; hence, $f$ is a required isometry defined on $%
Y.$ If $\mathbb{K}$ is non-spherically complete and $j\left( X\right)
+F_{Y}\neq Y$, then, by \cite[Proposition 2.1]{hilb}, $Y$ is an immediate
extension of $j\left( X\right) +F_{Y}$. Now, using Lemma \ref{ext-isom}, we
extend $f$ to the isometry defined on $Y$.
\end{proof}

The last part of the proof of Proposition \ref{Prop-ud} uses Lemma \ref%
{ext-isom} for the spherical completeness of the considered space $F$. In
fact it is enough to assume that $F$ contains a spherical completion of its
every finite-dimensional linear subspace. This observation suggests another
construction.

For each $g\in G$ set $I_{g}=\left\{ i_{g,1},i_{g,2},...\right\} $ (note
that $I_{g}$ is countable). For every $n\in \mathbb{N}$ set $F_{g}^{n}:=%
\left[ e_{i_{g,1}},...,e_{i_{g,n}}\right] ,$\ a finite-dimensional linear
subspace of $c_{0}\left( I_{u}\right) $ spanned by appropriate unit vectors.
By $\widehat{F_{g}^{n}}$ denote a spherical completion of $F_{g}^{n}$ such
that for fixed $g\in G$ we have $F_{g}^{n}\subset \widehat{F_{g}^{n}}\subset 
\widehat{c_{0}\left( I_{g}\right) }$ and $\widehat{F_{g}^{n-1}}\subset 
\widehat{F_{g}^{n}}$ if $n>1$.

Next, for every $g\in G$ define $F_{g}:=\bigcup_{n}\widehat{F_{g}^{n}}$. Let 
$E_{h}:=\bigoplus_{g\in G}\overline{F_{g}}.$

We are in a position to prove Theorem \ref{Eh}.

\begin{theorem}
\label{Eh}The space $E_{h}$ is of universal disposition for the class $%
\mathcal{U}_{FNA}$.
\end{theorem}

\begin{proof}
Let $\ Y$ be a finite-dimensional non-archimedean normed space, let $%
X\subset Y,$ and let $j:X\rightarrow E_{h}$ be an isometric embedding. We
prove that there exists an isometry $f:Y\rightarrow E_{h}$ such that $%
\left\vert \left\vert j-f|_{X}\right\vert \right\vert <1$. Then if we apply
Theorem \ref{Pro-iso}, the proof will be finished.

Set $t\in (0,1)$ and choose a $t-$orthogonal base $\left\{
x_{1},...,x_{n_{X}}\right\} $ of $X$. Let $z_{i}:=j\left( x_{i}\right) ,$ $%
i=1,...,n_{X}$. Then for each $i\in \left\{ 1,...,n_{X}\right\} $ there
exists a finite $G_{i}\subset G$, say $G_{i}=\left\{ g_{1},...,g_{m}\right\} 
$, and 
\begin{equation*}
z_{i}^{\prime }\in \bigoplus_{g\in G_{i}}\overline{F_{g}}
\end{equation*}%
for which 
\begin{equation*}
\left\vert \left\vert z_{i}^{\prime }-z_{i}\right\vert \right\vert <\frac{t}{%
2}\left\vert \left\vert z_{i}\right\vert \right\vert .
\end{equation*}

Fix $i\in \left\{ 1,...,n_{X}\right\} $. Then we can write $z_{i}^{\prime
}=z_{i,1}^{\prime }+...+z_{i,m}^{\prime },$ where $z_{i,k}^{\prime }\in 
\overline{F_{g_{k}}},$ $k=1,..,m$. But, then we can select $n_{i}$ and $%
w_{i,1}^{\prime },...,w_{i,m}^{\prime }$ such that $w_{i,k}^{\prime }\in 
\widehat{F_{g_{k}}^{n_{i}}}$, and $||w_{i,k}^{\prime }-z_{i,k}^{\prime }||<%
\frac{t}{2}||z_{i,k}^{\prime }||,$ $k=1,..,m.$ Denote $w_{i}:=w_{i,1}^{%
\prime }+...+w_{i,m}^{\prime }$. Then 
\begin{multline*}
\left\vert \left\vert w_{i}-z_{i}\right\vert \right\vert =\left\vert
\left\vert w_{i}-z_{i}^{\prime }+z_{i}^{\prime }-z_{i}\right\vert
\right\vert \leq \max \{\left\vert \left\vert w_{i}-z_{i}^{\prime
}\right\vert \right\vert ,\left\vert \left\vert z_{i}^{\prime
}-z_{i}\right\vert \right\vert \} \\
\leq \max \{\max_{k=1,...,m}\left\vert \left\vert w_{i,k}^{\prime
}-z_{i,k}^{\prime }\right\vert \right\vert ,\left\vert \left\vert
z_{i}^{\prime }-z_{i}\right\vert \right\vert \}<\frac{t}{2}\left\vert
\left\vert z_{i}\right\vert \right\vert .
\end{multline*}%
Hence, by Lemma \ref{L-ort}, $\left\{ w_{i}:i\in \left\{ 1,...,n_{X}\right\}
\right\} $\ is a $t-$orthogonal set in $\bigoplus_{g\in G_{0}}\widehat{%
F_{g}^{n_{0}}},$ where $G_{0}=G_{1}\cup ...\cup G_{n_{X}}$ and $n_{0}=\max
\left\{ n_{i}:i\in \left\{ 1,...,n_{X}\right\} \right\} $.

Define the map $f:X\rightarrow \bigoplus_{g\in G_{0}}\widehat{F_{g}^{n_{0}}}%
\subset E_{h}$ setting $f\left( x_{i}\right) :=w_{i},$ $i=1,...,n_{X}$. Then
for all $\lambda _{i}\in \mathbb{K}$ $\left( i=1,...,n_{X}\right) $ we have 
\begin{equation*}
f:\sum_{i=1}^{n_{X}}\lambda _{i}x_{i}\longmapsto \sum_{i=1}^{n_{X}}\lambda
_{i}w_{i}.
\end{equation*}
Consequently 
\begin{multline*}
\left\vert \left\vert f(\sum_{i=1}^{n_{X}}\lambda _{i}x_{i})\right\vert
\right\vert =\left\vert \left\vert \sum_{i=1}^{n_{X}}\lambda
_{i}w_{i}-\sum_{i=1}^{n_{X}}\lambda _{i}z_{i}+\sum_{i=1}^{n_{X}}\lambda
_{i}z_{i}\right\vert \right\vert =\left\vert \left\vert
\sum_{i=1}^{n_{X}}\lambda _{i}(w_{i}-z_{i})+\sum_{i=1}^{n_{X}}\lambda
_{i}z_{i}\right\vert \right\vert \\
=\left\vert \left\vert \sum_{i=1}^{n_{X}}\lambda _{i}z_{i}\right\vert
\right\vert =\left\vert \left\vert \sum_{i=1}^{n_{X}}\lambda
_{i}j(x_{i})\right\vert \right\vert =\left\vert \left\vert
\sum_{i=1}^{n_{X}}\lambda _{i}x_{i}\right\vert \right\vert
\end{multline*}%
since 
\begin{equation*}
\left\vert \left\vert \sum_{i=1}^{n_{X}}\lambda _{i}(w_{i}-z_{i})\right\vert
\right\vert <\frac{t}{2}\max_{i=1,...,n_{X}}\left\vert \left\vert \lambda
_{i}z_{i}\right\vert \right\vert
\end{equation*}%
and, as $\left\{ z_{1},...,z_{n_{X}}\right\} $ is $t-$orthogonal, we have 
\begin{equation*}
\left\vert \left\vert \sum_{i=1}^{n_{X}}\lambda _{i}z_{i}\right\vert
\right\vert \geq t\cdot \max_{i=1,...,n_{X}}\left\vert \left\vert \lambda
_{i}z_{i}\right\vert \right\vert .
\end{equation*}%
Thus $f$ is isometric. Observe that 
\begin{subequations}
\label{a11}
\begin{multline}
\left\vert \left\vert f-j\right\vert \right\vert =\sup_{x\in X}\frac{%
\left\vert \left\vert \left( f-j\right) \left( x\right) \right\vert
\right\vert }{\left\vert \left\vert x\right\vert \right\vert }=\sup_{\lambda
_{i}\in \mathbb{K}\text{ }\left( i=1,...,n_{X}\right) }\frac{\left\vert
\left\vert \sum_{i=1}^{n_{X}}\lambda _{i}\left( f-j\right)
(x_{i})\right\vert \right\vert }{\left\vert \left\vert
\sum_{i=1}^{n_{X}}\lambda _{i}x_{i}\right\vert \right\vert } \\
\leq \sup_{\lambda _{i}\in \mathbb{K}\text{ }\left( i=1,...,n_{X}\right) }%
\frac{\max_{i=1,...,n_{X}}\left\vert \left\vert \lambda _{i}\left(
w_{i}-z_{i}\right) \right\vert \right\vert }{t\cdot
\max_{i=1,...,n_{X}}\left\vert \left\vert \lambda _{i}x_{i}\right\vert
\right\vert }  \notag \\
\leq \sup_{\lambda _{i}\in \mathbb{K}\text{ }\left( i=1,...,n_{X}\right) }%
\frac{\frac{t}{2}\cdot \max_{i=1,...,n_{X}}\left\vert \left\vert \lambda
_{i}x_{i}\right\vert \right\vert }{t\cdot \max_{i=1,...,n_{X}}\left\vert
\left\vert \lambda _{i}x_{i}\right\vert \right\vert }<1.
\end{multline}

Now, we extend $f$ on $Y.$ We argue similarly as in the proof of the
previous theorem. Choose $\left\{ u_{1},...,u_{m_{X}},...,u_{m_{Y}}\right\} $%
, a maximal orthogonal set in $Y$ such that $\left\{
u_{1},...,u_{m_{X}}\right\} $ is a maximal orthogonal set in $X$. \ By Lemma %
\ref{nowy}\ we get $F_{Y}\perp j(X)$, where $F_{Y}=\left[
u_{m_{X}+1},...,u_{m_{Y}}\right] $.

Denote $v_{k}:=f\left( u_{k}\right) $ for $k=1,...,m_{X}$, and for each $%
k\in \left\{ m_{X+1},...,m_{Y}\right\} $ choose $g_{k}\in G$ and $%
i_{g_{k},n_{k}}\in I_{g_{k}},$ denoting for simplicity $%
j_{k}:=i_{g_{k},n_{k}}$,\ such that $\left\vert \left\vert
e_{j_{k}}\right\vert \right\vert _{u}=\left\vert \left\vert \lambda
_{k}u_{k}\right\vert \right\vert $ (for some $\lambda _{k}\in \mathbb{K}$)
and 
\end{subequations}
\begin{equation*}
e_{j_{k}}\perp \lbrack v_{1},...,v_{m_{X}},e_{i_{m_{X+1}}},...,e_{j_{k-1}}].
\end{equation*}%
Then, setting $f\left( u_{k}\right) :=e_{j_{k}},$ where $k=n_{X+1},...,n_{Y}$%
, we extend $f$ on $X+F_{Y}$. Let $p_{0}=\max \left\{
n_{k}:k=m_{X+1},...,m_{Y}\right\} $ and $G_{1}=\left\{
g_{k}:k=m_{X+1},...,m_{Y}\right\} $. The map $f$ is an isometry and $f\left(
X+F_{Y}\right) \subset H$, where 
\begin{equation*}
H=\bigoplus_{g\in G_{0}\cup G_{1}}\widehat{F_{g}^{\max \{n_{0},p_{0}\}}}.
\end{equation*}

If $\mathbb{K}$ is spherically complete, the proof is completed, since $%
\left\{ u_{1},...,u_{n_{Y}}\right\} $ is an orthogonal base of $Y$ by \cite[%
Lemma 5.5 and Theorem 5.15]{Rooij}. This shows that $f$ is a required
isometry defined on $Y.$ Assume that $\mathbb{K}$ is non-spherically
complete. Since $Y$ is an immediate extension of $X+F_{Y}$ and $H$ is
spherically complete as a finite direct sum of spherically complete spaces
(see \cite[4A]{Rooij}), we apply Lemma \ref{ext-isom} to extend $f$ to the
isometry defined on $Y$. Since $\left\vert \left\vert j-f|_{X}\right\vert
\right\vert <1$, as we proved above, we apply Theorem \ref{Pro-iso}. The
proof is finished.
\end{proof}

This yields the following interesting

\begin{corollary}
\label{Ehh} $\left( 1\right) $ If $\mathbb{K}$ is spherically complete and $%
\left( 0,\infty \right) $ is the union of at most countably many cosets of $%
\left\vert \mathbb{K}^{\ast }\right\vert $, the isometrically universal
space $E_{h}$ for the class of non-archimedean Banach spaces of countable
type is also of universal disposition for the class $\mathcal{U}_{FNA}$.

$\left( 2\right) $ There exist non-archimedean Banach spaces of universal
disposition for the class $\mathcal{U}_{FNA}$ which are not isometrically
isomorphic.
\end{corollary}

\begin{proof}
Recall that if $\mathbb{K}$ is spherically complete, then every
finite-dimensional normed space over $\mathbb{K}$ is spherically complete
(see \cite[Theorem 4.2 and Corollary 4.6 ]{Rooij}). Thus, $E_{h}=\left(
c_{0}\left( I_{u}\right) ,\left\vert \left\vert .\right\vert \right\vert
_{u}\right) =E_{u}$. The remaining part of the proof of $\left( 1\right) $
follows from Proposition \ref{P-univers}. To prove $\left( 2\right) $
consider the spaces $E_{h}$ and $\widehat{E_{u}}$ (by Theorems \ref{Prop-ud}
and \ref{Eh} both are spaces of universal disposition for the class $%
\mathcal{U}_{FNA}$) assuming that $\mathbb{K}$ is spherically complete. Then 
$E_{h}=E_{u}$. Since $\left\vert \left\vert E_{u}^{\ast }\right\vert
\right\vert _{u}$ is a dense subset of $\left( 0,\infty \right) $, we can
choose $\left\{ i_{1},i_{2},...\right\} \subset I_{u}$ such that \ $1\geq
\left\vert \left\vert e_{i_{1}}\right\vert \right\vert _{u}>\left\vert
\left\vert e_{i_{2}}\right\vert \right\vert _{u}>...>\frac{1}{2}$ (where $%
e_{i_{k}}$ are unit vectors of $E_{u}$). Set $x_{n}:=\sum_{k=1}^{n}e_{i_{k}}$
$\left( n\in \mathbb{N}\right) $. Then, the balls $V_{n}:=\left\{ x\in
E_{u} :\left\vert \left\vert x-x_{n}\right\vert
\right\vert _{u}\leq \left\vert \left\vert e_{i_{n+1}}\right\vert
\right\vert _{u}\right\} $\ form a shrinking sequence in $E_{u}$.
But $\bigcap_{n\in \mathbb{N}}V_{n}=\varnothing $; hence, $E_{u}$ is not
spherically complete. Clearly $\widehat{E_{u}}$, as a spherical completion
of $E_{u}$, is spherically complete, thus we imply that $E_{h}$ (=$E_{u}$)
and $\widehat{E_{u}}$ \ are not isometrically isomorphic.
\end{proof}

\begin{remark}
\label{rem-ud}Note that (see Remark \ref{rem-univ}) if $\mathbb{K}$ is
non-spherically complete then the space $\ell^{\infty }$ is not of universal
disposition for the class $\mathcal{U}_{FNA}$. Indeed, take $Y=\mathbb{K}%
_{v}^{2}$, $e_{1}=\left( 1,0,0,...\right) \in \ell^{\infty }$ and define the
isometric embedding $i:\left[ e_{1}\right] \rightarrow \mathbb{K}%
_{v}^{2}:e_{1}\longmapsto \left( 1,0\right) $. On the other hand, by \cite[Proposition 3.2]%
{AK-2011-linf}, every two-dimensional linear subspace of $\ell^{\infty }$ containing $e_{1}$ has two non-zero orthogonal
elements. Thus, there is no isometric embedding $f:Y\rightarrow \ell^{\infty }$
such that $f(1,0) = e_{1}$.
\end{remark}

\end{document}